\documentclass[10pt,reqno]{amsart}
\usepackage[T1]{fontenc}
\usepackage{color, amssymb, amsmath}
\usepackage{cite}
\usepackage{xcolor}
\definecolor{MyLinkColor}{rgb}{0,0,0.4}

\newtheorem{thm}{Theorem}[section]

\newtheorem{cor}[thm]{Corollary}
\newtheorem{prop}[thm]{Proposition}

\theoremstyle{definition}

\theoremstyle{remark} 
\newtheorem{rems}[thm]{Remarks}
\newtheorem{rem}[thm]{Remark}

\numberwithin{equation}{section}

\newcommand{\N}{\mathbb N}
\newcommand{\R}{\mathbb R}
\newcommand{\C}{\mathbb C}
\newcommand{\K}{\mathbb K}
\newcommand{\bA}{\mathbb A}
\newcommand{\bB}{\mathbb B}

\newcommand{\kL}{\mathcal L}
\newcommand{\ve}{\varepsilon}

\newcommand{\wh}{\widehat}
\newcommand{\ov}{\overline}
\newcommand{\rd}{\mathrm{d}}
\newcommand{\p}{\partial}
\newcommand{\bqn}{\begin{equation}}
\newcommand{\eqn}{\end{equation}}

\numberwithin{equation}{section} 
\usepackage[colorlinks=true,linkcolor=MyLinkColor,citecolor=MyLinkColor]{hyperref}

\begin{document}

\title[The principle of linearized stability for quasilinear evolution equations]{On the principle of linearized stability for  quasilinear evolution equations in time-weighted spaces}

\author{Bogdan--Vasile Matioc}
\address{Fakult\"at f\"ur Mathematik, Universit\"at Regensburg,   93053 Regensburg, Deutschland.}
\email{bogdan.matioc@ur.de}

\author{Lina Sophie Schmitz}
\address{Leibniz Universit\"at Hannover\\
Institut f\"ur Angewandte Mathematik\\
Welfengarten 1\\
30167 Hannover\\
Germany}
\email{schmitz@ifam.uni-hannover.de}

\author{Christoph Walker}
\address{Leibniz Universit\"at Hannover\\
Institut f\"ur Angewandte Mathematik\\
Welfengarten 1\\
30167 Hannover\\
Germany}
\email{walker@ifam.uni-hannover.de}

\begin{abstract}
 Quasilinear (and semilinear) parabolic problems of the form $v'=A(v)v+f(v)$ with strict inclusion $\mathrm{dom}(f)\subsetneq \mathrm{dom}(A)$ of the domains of the function $v\mapsto f(v)$ and the quasilinear part $v\mapsto A(v)$ are considered in the framework of time-weighted function spaces.
This allows one to establish the principle of linearized stability in intermediate spaces lying  between $\mathrm{dom}(f)$ and $\mathrm{dom}(A)$ and  yields a greater flexibility with respect to the phase space for the evolution. 
 In applications to differential equations such intermediate spaces may correspond  to critical spaces exhibiting a scaling invariance. 
Several examples are provided to demonstrate the  applicability of the results.
\end{abstract}


\maketitle

\section{Introduction}\label{Sec:1}

The principle of linearized stability is a widely recognized method  for various nonlinear parabolic evolution equations to derive stability or instability properties of an equilibrium from the real parts of the spectral points of the linearization at the equilibrium. Extensive research has been conducted on this subject under different assumptions and using various techniques,
see e.g. \cite{DaPL88,D89,G88,Lu85,L95,PF81,P02,PSZ09, PSW18, PSZ09b,MW_MOFM20}, though this list is by no means exhaustive.

 In this research we focus on quasilinear parabolic problems
\begin{equation}\label{EE}
 v'=A(v)v+f(v)\,,\quad t>0\,,\qquad v(0)=v^0\,,
\end{equation}
(and the corresponding semilinear counterparts) assuming a strict inclusion $\mathrm{dom}(f)\subsetneq \mathrm{dom}(A)$ for the domains of the function $v\mapsto f(v)$ and the quasilinear part $v\mapsto A(v)$. We shall establish herein the principle of linearized stability in phase spaces for the initial values that lie between $\mathrm{dom}(f)$ and $\mathrm{dom}(A)$ enabling greater flexibility in applications. 
For this purpose we rely on previous results \cite{MW_PRSE,MRW25} on the well-posedness of~\eqref{EE} in time-weighted function spaces  and extend the principle of linearized stability \cite{MW_MOFM20} (for the case  $\mathrm{dom}(f)= \mathrm{dom}(A)$) to the present setting.

To be more precise, let $E_0$ and $E_1$ be Banach spaces over $\K\in \{\R,\C\}$ with continuous and dense embedding 
\[
E_1 \stackrel{d}{\hookrightarrow} E_0\,.
\]
 Given $\theta\in (0,1)$, we fix an admissible interpolation functor $(\cdot,\cdot)_\theta$ of exponent $\theta$ (see \cite[I.Sections~2.1, 2.11]{LQPP}) and set $E_\theta:= (E_0,E_1)_\theta$ and  $\|\cdot\|_\theta:=\|\cdot\|_{E_\theta}$. Then
\[
E_1\stackrel{d}{\hookrightarrow} E_\theta  \stackrel{d}{\hookrightarrow} E_0\,.
\]
Moreover,  we fix numbers
\begin{subequations}\label{AS1}
\begin{equation}\label{a1a}
0<\gamma< \beta< \xi<1\,, \qquad   q\ge 1\,,
\end{equation}
and assume for the quasilinear part in \eqref{EE}  that
\begin{equation}\label{a1b}
A\in C^{1-}\big(O_\beta,\mathcal{H}(E_1,E_0) \big)\,,
\end{equation}
where
\begin{equation}\label{a1c}
 \emptyset\not= O_\beta\, \text{ is an open subset of } E_\beta\,.
\end{equation}
By $\mathcal{H}(E_1,E_0)$  we denote the open subset of the bounded linear operators $\kL(E_1,E_0)$  
consisting of generators of strongly continuous analytic semigroups on $E_0$.
The semilinear part  $f:O_\xi\to E_\gamma$ is assumed to belong to the class of locally Lipschitz continuous mappings from the   open subset~${O_\xi:=O_\beta\cap E_\xi}$ of $E_\xi$ into $E_\gamma$.
 More  precisely, we assume  that for each~$R>0$ there is~$c(R)>0$  such that,
 for all $w,v\in O_\xi\cap \ov{\mathbb{B}}_{E_\beta}(0,R)$,
\begin{equation}\label{AS2x}
\|f(w)-f(v)\|_{\gamma}\le  c(R)\big[1+\|w\|_{\xi}^{q-1}+\|v\|_{\xi}^{q-1}\big]\big[\big(1+\|w\|_{\xi}+\|v\|_{\xi}\big) \|w-v\|_{\beta}+\|w-v\|_{\xi}\big]\,.
\end{equation}
 Furthermore, we define a critical value  $\alpha_{\rm  crit}<\xi$ by setting
\begin{equation}\label{X}
\alpha_{\rm  crit}:=\frac{q\xi-1-\gamma}{q-1} \quad \text{ if } q>1\,,\qquad \alpha_{\rm  crit}:= -\infty \quad \text{ if } q=1\,,
\end{equation}
and consider  initial values $v^0\in O_\alpha:=O_\beta\cap E_\alpha$ with  
\begin{equation}\label{XX}
 \alpha_{\rm  crit}\le\alpha\in (\beta,\xi)\,.
\end{equation}
\end{subequations}
 The well-posedness theory for \eqref{EE} developed in \cite{MRW25} for the critical case $\alpha = \alpha_{\rm crit} \in (\beta, \xi)$ relies 
on the assumption  that there  is an interpolation functor $\{\cdot,\cdot\}_{\alpha/\xi}$ of  
exponent~$\alpha/\xi$ and, for $\eta \in \{\alpha, \beta, \xi\}$, there are interpolation functors~${\{\cdot,\cdot\}_{\gamma/\eta}}$ of  
exponents~$\gamma/\eta$ such that
\begin{equation}\label{AS2y}
   E_\alpha\doteq \{E_0,E_\xi\}_{\alpha/\xi} \,, \qquad  E_\gamma\doteq \{E_0,E_\eta\}_{\gamma/\eta} \,, \quad \eta\in\{\alpha,\beta,\xi\}\,.
\end{equation} 
Assumption~\eqref{AS2y} is not really restrictive in applications:
\begin{rem}\label{RR1}
Assumption~\eqref{AS2y} is automatically satisfied if, for each~$\theta\in\{\gamma,\beta,\alpha,\xi\}$, the interpolation functor~${(\cdot,\cdot)_\theta}$ is
 chosen to be  always either the complex interpolation functor~$[\cdot,\cdot]_\theta$,  or the continuous interpolation functor~${(\cdot,\cdot)_{\theta,\infty}^0}$,
  or the real interpolation functor~$(\cdot,\cdot)_\theta=(\cdot,\cdot)_{\theta,p}$  with parameter $p\in [1,\infty]$.
This is a consequence of the reiteration theorems for these functors, see  e.g.~\cite[I.Remarks~2.11.2~(b)]{LQPP}.
\end{rem}

 It is also  worth emphasizing that \eqref{AS2x} is satisfied  in the particular case when
\begin{equation}\label{chris}
\|f(w)-f(v)\|_{\gamma}\le  c \big[\|w\|_{\xi}^{q-1}+\|v\|_{\xi}^{q-1}\big] \|w-v\|_{\xi} \,,\qquad w,\,v\in O_\xi\,.
\end{equation}
  The special case \eqref{chris} of \eqref{AS2x} frequently arises in applications; see, for instance, the examples provided in Section~\ref{Sec:5} and \cite[Lemma~4.1]{MW_PRSE}.

 Let us point out with respect to the parameters in \eqref{a1a} that the parameter $q\geq1$ measures the growth of the nonlinearity~$f$ with respect to the $E_\xi$-terms, and its value is  usually fixed
 and  cannot be adjusted freely.
In certain applications, such as considered in Examples~\ref{Exam1} and~\ref{Exam2}  below, there may be, however, some flexibility with respect to the choice of the 
 interpolation exponents $\,\beta, \gamma$, and $\xi$. 
These parameters then define the range of~${\alpha\ge\alpha_{\rm  crit}}$ for the phase space $E_\alpha$ with critical value $\alpha=\alpha_{\rm  crit}$. 
In certain applications, the critical space~$E_{\alpha_{\rm  crit}}$  can also be identified due to scaling invariance properties (see Example~\ref{Exam2}).

For simplicity,  we only consider herein the situation that $\alpha<\xi$,  as the case $\alpha > \xi$ is treated in~\cite{MW_MOFM20}\footnote{ The case $\alpha=\xi$ could, in principle, be incorporated into the present analysis, but requires a slightly different functional analytic setting since certain statements (for instance \eqref{solb} below) no longer hold. 
}.  
We emphasize that  thus $E_\xi \hookrightarrow E_\alpha\hookrightarrow E_\beta$  in the setting herein and  hence, the semilinear part  $f$, being defined on $O_\xi$, need not be defined on the phase space 
$E_\alpha$ and may require higher regularity than the quasilinear part $A$.
 Compared to the classical setting~\cite{Amann_Teubner, D89, MW_MOFM20}, where both $A$ and $f$ are assumed to be locally Lipschitz continuous on $O_\xi$,  
and well-posedness and stability  of equilibria are established for initial data $v^0 \in O_\alpha$ with~$\alpha > \xi$,
 we shall use condition~\eqref{AS2x} (or~\eqref{Vor3} in the semilinear case)  
 to  establish similar results for initial data $v^0 \in O_\alpha$ with~$\alpha < \xi$,  
 additionally assuming~\eqref{X}--\eqref{XX} (or~\eqref{Vor2}--\eqref{XXs} for the semilinear case).  
  In fact,   condition~\eqref{AS2x} (or~\eqref{Vor3}) implies precise quantitative estimates on the underlying semigroups  enabling the use of time-weighted spaces in the analysis of the evolution problem, 
which effectively capture the (singular) behavior of solutions at $t=0$ in $E_\xi$.
\\

We first recall the well-posedness of the quasilinear problem~\eqref{EE}.

\subsection*{Well-Posedness}
Existence and uniqueness of solutions to  problem \eqref{EE}  have been derived in \cite[Theorem~1.1]{MW_PRSE}  for $\alpha>\alpha_{\rm  crit}$ assuming~\eqref{AS1},  respectively  in \cite[Theorem~1.2]{MRW25} 
for the critical case $\alpha=\alpha_{\rm  crit}$ assuming~\eqref{AS1}-\eqref{AS2y}. 
The key ingredient to establish well-posedness for initial values $v^0$ in the phase space~$E_\alpha$ is the use of time-weighted spaces of continuous functions $v:(0,T]\to E_\xi$ satisfying
\begin{equation*} 
\sup_{t\in (0,T]}t^\mu\|v(t)\|_{\xi}<\infty\qquad \text{and}\qquad \lim_{t\to 0} t^\mu\|v(t)\|_{\xi}=0
\end{equation*}
for some $\mu\ge \xi-\alpha$, which are
adapted to the regularizing effects of the parabolic operator $A$.

 The relevant aspects of the results from \cite{MW_PRSE,MRW25} that are essential for our present purposes are summarized in Theorem~\ref{T0}. 
 To unify these aspects into a single statement applicable to both  critical and noncritical regimes, we assume throughout~\eqref{AS1}--\eqref{AS2y}. This approach provides a slight refinement of the well-posedness result established in \cite[Theorem~1.1]{MW_PRSE}, as detailed below.

\begin{thm}[Well-Posedness]\label{T0}
 Assume~\eqref{AS1}-\eqref{AS2y}. 
 Then, given any~${v^0\in O_\alpha}$, the Cauchy problem \eqref{EE} possesses a  unique maximal strong  solution
\begin{subequations}\label{sol}
\begin{equation} \label{sola}
\begin{aligned}
 v(\cdot;v^0)&\in C^1\big((0,t^+(v^0)),E_0\big)\cap C\big((0,t^+(v^0)),E_1\big)\cap  C \big([0,t^+(v^0)),O_\alpha\big)\cap C^{\alpha-\beta}\big([0, t^+(v^0)), E_\beta\big)
\end{aligned}
\end{equation}
with $ t^+(v^0)\in(0,\infty]$, such  that 
 \begin{equation} \label{solb}
\lim_{t\to 0} t^{\xi-\alpha}\|v(t;v^0)\|_{E_\xi}=0\,.
\end{equation}  
\end{subequations} 
Moreover, if  $v^0\in O_\alpha$ is such that $t^+(v^0)<\infty$ and  $v(\cdot;v^0):[0,t^+(v^0))\to E_\beta$ is uniformly H\"older continuous, then
\begin{equation}\label{blow}
 \underset{t\nearrow t^+(v^0)}{\limsup}\|f(v(t;v^0))\|_{0}=\infty\quad\text{ or }\quad \lim_{t\nearrow t^+(v^0)}\mathrm{dist}_{E_\beta}\big( v(t;v^0),\partial O_\beta\big)=0\,.
\end{equation}
\end{thm}

\begin{proof}
 The existence and uniqueness of  a maximal strong solution is established in  \cite[Theorem 1.2]{MRW25} for the critical value~$\alpha_{\rm  crit}=\alpha\in (\beta,\xi)$.
Moreover,  for the noncritical case~$\alpha_{\rm  crit}<\alpha\in (\beta,\xi)$, it is proven in~\cite[Theorem~1.1]{MW_PRSE}  (without assuming~\eqref{AS2y}) that \eqref{EE} has a unique maximal strong 
solution which satisfies~\eqref{sola} together with
 \begin{equation} \label{solb'}
\lim_{t\to 0} t^{\mu}\|v(t;v^0)\|_{E_\xi}=0
\end{equation}
for all $\mu>\xi-\alpha$ (instead of the stronger property~\eqref{solb}).
This solution  fulfills the variation-of-constants  formula  
\begin{equation*} 
v(t)=U_{A(v)}(t,0)v^0+ \int_0^t U_{A(v)}(t,\tau) f(v(\tau))\,\rd \tau\,,\qquad t\in(0,t^+(v^0))\,,
\end{equation*}
 where $U_{A(v)}(\cdot,\cdot)$ is the evolution operator associated with $A(v(\cdot))$.
Assuming~\eqref{AS2y}, it readily follows  that~${t^{\xi-\alpha} U_{A(v)}(t,0)v^0\to  0}$ in $E_\xi$ as $t\to0$, see e.g. the proof of \cite[Proposition~2.1]{MRW25}. 
Moreover, arguing as in the proof of \cite[Proposition~2.1, Eq. (2.18)]{MW_PRSE},  we infer  from \eqref{solb'} that
 \[
 t^{\xi-\alpha}\int_0^t U_{A(v)}(t,\tau) f(v(\tau))\,\rd \tau\underset{t\to 0}\longrightarrow 0 \qquad  \text{in $E_\xi$}\,,
 \]
 which proves  \eqref{solb}.
 
Finally, the blow-up criterion \eqref{blow} is provided for~$\alpha_{\rm  crit}=\alpha\in (\beta,\xi)$  in \cite[Theorem 1.2~(v)~(b)]{MRW25}.
The proof of~\eqref{blow} in the case when~$\alpha_{\rm  crit}<\alpha\in (\beta,\xi)$ is identical to that of \cite[Theorem 1.2~(v)~(b)]{MRW25} and therefore we  omit herein the details.
\end{proof}

\subsection*{Exponential Stability}
 We shall now present the principle of linearized stability in the phase space~$E_\alpha$ for the  range of exponents $\alpha$ specified  in~\eqref{XX}. 
To this end, let
\begin{subequations}\label{AS4}
\begin{equation}\label{AS3}
v_*\in O_1:= O_\beta\cap E_1 \qquad\text{with}\qquad  A(v_*)v_*+f(v_*)=0
\end{equation}
be an equilibrium solution to \eqref{EE}. 
In order to derive stability properties for $v_*$, we further assume that 
\begin{equation}\label{a4a}
f:O_\xi\rightarrow E_0 \quad\text{and}\quad A(\cdot)v_*: O_\xi\rightarrow  E_0 \quad\text{are Fr\'echet differentiable at}\ v_* 
\end{equation}
with Fr\'echet derivatives $\p f(v_*)\in\kL(E_\xi,E_0)$ and $ (\p A(v_*)[\cdot])v_*\in\kL( E_\xi,E_0)$, respectively.
Moreover, we assume that the linearized operator
 \[
 \bA:=  A(v_*)+(\p A(v_*)[\cdot])v_*+\p f(v_*)\in\kL(E_1,E_0)
 \]
  has a negative spectral bound, that is,
\begin{equation}\label{a4e}
-\omega_0:=s(\bA):=\sup\left\{\mathrm{Re}\, \lambda\,:\, \lambda \in \sigma(\bA)\right\}<0\,.
\end{equation}
\end{subequations}
Our assumptions will allow us to rely on the previously established principle of linearized stability from \cite{MW_MOFM20}.
 In fact, we shall prove the following result regarding the asymptotic exponential stability of the equilibrium~$v_*$ and thereby extend the corresponding result of \cite{MW_MOFM20} to the present framework.

\begin{thm}[Stability] \label{T1}
 Assume~\eqref{AS1}, \eqref{AS2y},  and~\eqref{AS4}.
Then, the equilibrium~$v_*$ is asymptotically exponentially stable
 in $E_\alpha$. More  precisely, given any $ \omega\in (0,\omega_0)$, there  exist $\ve_0>0$ and~${M\ge 1}$ such that, 
 for each $v^0\in \ov\bB_{E_\alpha}(v_*,\ve_0)$, the solution to \eqref{EE} exists globally in time and
\begin{equation}\label{stable}
\|v(t;v^0)-v_*\|_\alpha +t^{\xi-\alpha}\|v(t;v^0)-v_*\|_\xi \le M e^{-\omega t}\|v^0-v_*\|_\alpha\,,\qquad t> 0\,.
\end{equation}
\end{thm}

 The interpolation assumptions~\eqref{AS2y}  are not required in any case:

\begin{rem}\label{RRa1}
In the noncritical case $  \alpha_{\rm  crit}<\alpha\in  (\beta,\xi)$  one may actually drop the assumptions~\eqref{AS2y} in Theorem~\ref{T1}.
The claim remains valid (with a similar proof) provided  that the estimate  \eqref{stable} is   replaced by
\begin{equation*}
\|v(t;v^0)-v_*\|_\alpha + t^{\mu}\|v(t;v^0)-v_*\|_\xi\le M e^{-\omega t}\|v^0-v_*\|_\alpha\,,\qquad t > 0\,,
\end{equation*} 
for some fixed (but arbitrary)  $\mu>\xi-\alpha$. 
Since in many applications \eqref{AS2y} are automatically satisfied (see Remark~\ref{RR1}), 
we chose to present  the (slightly) less general,  but more concise,  result in Theorem~\ref{T1}.
\end{rem}

  \subsection*{Semilinear Evolution Equations}   The previous results hold, of course, also for the semilinear evolution problem    
 \begin{equation}\label{SEP}
v'=Av+f(v)\,,\quad t>0\,,\qquad v(0)= v^0\,,
\end{equation}
 with 
\begin{subequations}\label{Vor}
\begin{equation}\label{Vor1}
A\in \mathcal{H}(E_1,E_0)\,.
\end{equation} 
 However, for this particular case, we present a sharper version of  the exponential stability result in Theorem~\ref{T1}. 
Indeed, we now only assume that  
 \begin{equation}\label{Vor2}
  0\leq \gamma<   \xi\leq  1\,,\qquad (\gamma,\xi)\neq (0,1)\,,\qquad  q\ge 1\,,
\end{equation}
 and define again  $\alpha_{\rm  crit}<\xi$ by
\begin{equation}\label{Xs}
\alpha_{\rm  crit}:=\frac{q\xi-1-\gamma}{q-1} \quad \text{ if } q>1\,,\qquad \alpha_{\rm  crit}:= -\infty \quad \text{ if } q=1\,.
\end{equation}
We assume that
\begin{equation}\label{XXs}
 \text{either}\qquad \alpha_{\rm  crit}<\alpha\in [\gamma,\xi)\qquad \text{or}\qquad \alpha_{\rm  crit}=\alpha\in (\gamma,\xi)\,,
\end{equation}
and  let $O_\alpha$ be  an arbitrary open subset  of~$E_\alpha$. 
The semilinearity  $f:O_\xi:=O_\alpha\cap E_\xi\to E_\gamma$   is again assumed to be locally Lipschitz continuous in the sense that   for each~${R>0}$ there is a  constant~$c(R)>0$ such that 
 \begin{equation}\label{Vor3}
\|f(w)-f(v)\|_{\gamma}\le   c(R)\big[1+\|w\|_{\xi}^{q-1}+\|v\|_{\xi}^{q-1}\big]\big[\big(1+\|w\|_{\xi}+\|v\|_{\xi}\big) \|w-v\|_{\alpha}+\|w-v\|_{\xi}\big] 
\end{equation}
\end{subequations}
for all $w,\,v\in  O_\xi \cap\ov{\mathbb{B}}_{E_\alpha}(0,R)$.

 The well-posedness of \eqref{EE} is established in \cite{MRW25} for the critical case~$\alpha=\alpha_{\rm  crit}\in(\gamma,\xi)$ under the assumption that, when $\xi<1$,
 there exists an interpolation functor~$\{\cdot,\cdot\}_{\alpha/\xi}$ of exponent $\alpha/\xi$,  and, when~${\gamma>0}$, that for~${\eta\in\{\alpha,\xi\}\setminus\{1\}}$ 
there are  interpolation functors~$\{\cdot,\cdot\}_{\gamma/\eta}$ of exponent~$\gamma/\eta$ such that
\begin{equation}\label{Vorcc}
   E_\alpha\doteq \{E_0,E_\xi\}_{\alpha/\xi}\,\,\, \text{if $\xi<1$}  \,, \qquad  E_\gamma\doteq \{E_0,E_\eta\}_{\gamma/\eta} \,\,\, \text{if $\gamma>0$\,, 
   \ $\eta\in\{\alpha,\xi\}\setminus\{1\}$\,.}
\end{equation}

Under these assumptions the Cauchy problem \eqref{SEP} is locally well-posed in $O_\alpha$.
Similar to the quasilinear problem~\eqref{EE}, 
we also assume in the noncritical case the interpolation property \eqref{Vorcc}. 
This leads to a unified and slightly refined well-posedness result for \eqref{SEP}, compared to \cite[Theorem 1.3]{MRW25}.

\begin{thm}[Well-Posedness]\label{T0'}
 Assume~\eqref{Vor}-\eqref{Vorcc}.
  Then, given any $v^0\in O_\alpha$, the  semilinear Cauchy   problem~\eqref{SEP} possesses a  unique maximal strong  solution
\begin{subequations}\label{solse}
\begin{equation} 
\begin{aligned}
 v(\cdot;v^0)&\in C^1\big((0,t^+(v^0)),E_0\big)\cap C\big((0,t^+(v^0)),E_1\big)\cap  C \big([0,t^+(v^0)),O_\alpha\big)
\end{aligned}
\end{equation}
with $ t^+(v^0)\in(0,\infty]$, such  that 
 \begin{equation} 
\lim_{t\to 0} t^{\xi-\alpha}\|v(t;v^0)\|_{E_\xi}=0\,.
\end{equation} 
\end{subequations} 
Moreover, if  $v^0\in O_\alpha$ is such that $t^+(v^0)<\infty$,  then 
\begin{equation}\label{blowse1}
 \underset{t\nearrow t^+(v^0)}{\limsup}\|f(v(t;v^0))\|_{0}=\infty\qquad\text{ or }\qquad \lim_{t\nearrow t^+(v^0)}\mathrm{dist}_{E_\alpha}\big( v(t;v^0),\partial O_\alpha\big)=0\,.
\end{equation}
\end{thm}
\begin{proof}
This result is proven in \cite[Theorem 1.3]{MRW25}  for the critical case   $\alpha=\alpha_{\rm  crit}\in (\gamma,\xi)$. 
For~$\alpha_{\rm  crit}<\alpha\in [\gamma,\xi)$ the claim follows by arguing along the lines of the proof of \cite[Theorem 1.3]{MRW25} (we omit therefore the details).
For an alternative proof in the case $\alpha_{\rm  crit}<\alpha\in [\gamma,\xi)$ under the restriction $O_\alpha=E_\alpha$ 
(but with a stronger blow-up criterion than in \eqref{blowse1}) we refer to \cite[Theorem 1.2]{MW_PRSE}.
\end{proof}

Let again
   \begin{subequations}\label{MM}
\begin{equation}\label{MM1}
v_*\in O_1:= O_\alpha\cap E_1 \qquad\text{with}\qquad A v_*+f(v_*)=0
\end{equation}
be an equilibrium solution to \eqref{SEP} such that 
\begin{equation}\label{MM2}
f:O_\xi\rightarrow E_0  \qquad\text{is Fr\'echet differentiable at $v_*$}
\end{equation}
 with Fr\'echet derivative $\p f(v_*)\in\kL(E_\xi,E_0)$.
For the linearized operator
  \[
 \bA:=  A+\p f(v_*)\in\kL(E_1,E_0)
 \]
we assume a negative spectral bound, that is,
\begin{equation}\label{MM5}
-\omega_0:=s(\bA):=\sup\left\{\mathrm{Re}\, \lambda\,:\, \lambda \in \sigma(\bA)\right\}<0\,.
\end{equation}
\end{subequations}

If $0<\gamma<\xi< 1$ in \eqref{Vor}, the principle of linearized stability from \cite{MW_MOFM20} can be applied (similarly as in the quasilinear case) to establish the exponential stability of the equilibrium solution~$v_*$ to the semilinear evolution problem~\eqref{SEP}. In contrast, when $\gamma = 0$ or $\xi = 1$ in \eqref{Vor}, this approach is no longer applicable. Nonetheless, exponential stability of the equilibrium~$v_*$ can still be obtained under additional, yet mild, structural assumptions.
More precisely, we assume  that there exist
\begin{subequations}\label{VorMM}
\begin{equation}\label{MM3}
\gamma_*\in[0,\gamma] \,, \qquad   q_*>1\,,
\end{equation}
together with constants $r_*,\,c_*>0$ such that
\begin{equation}\label{MM4}
\|f(w+v_*)-f(v_*)-\p f(v_*)w \|_{\gamma_*}\le c_* \|w\|_\xi^{q_*}\,,\qquad w\in \wh O_\xi\cap{\ov{\mathbb{B}}_{E_\alpha}(0,r_*)}\,,
\end{equation}
where $\wh O_\xi := O_\xi - v_*$. Concerning the choice of $q_*$ and $\gamma_*$  see also Remark~\ref{RRa2'} below.
\end{subequations}

\begin{thm}[Stability]\label{Tsemi}
 Assume \eqref{Vorcc}, \eqref{MM}, and  either
\begin{itemize}  
\item[(a)] \eqref{Vor} with $0<\gamma<\xi< 1$;
\end{itemize}
or 
\begin{itemize}  
\item[(b)] \eqref{Vor}  with $\gamma=0$ or $\xi=1$, and \eqref{VorMM}. In this case set
\begin{equation*}
\alpha_{\rm  crit}^*:=\frac{q_*\xi-1-\gamma_*}{q_*-1}<\xi
\end{equation*}
and assume additionally that  $\alpha\ge \alpha_{\rm  crit}^*$ with strict inequality  $\alpha>\alpha_{\rm  crit}^*$ 
when~$\gamma_*\in (0,\gamma)$  or~$\alpha=\gamma$.
\end{itemize}
Then, the equilibrium solution~$v_*$ to \eqref{SEP} is asymptotically exponentially stable in $E_\alpha$. That is, given any~${\omega\in (0,\omega_0)}$, there are~$\ve_0>0$ and~${M\ge 1}$ such that, 
 for each $v^0\in \ov\bB_{E_\alpha}(v_*,\ve_0)$, the solution to~\eqref{SEP} exists globally in time and
\begin{equation}\label{stable'}
\|v(t;v^0)-v_*\|_\alpha + t^{\xi-\alpha}\|v(t;v^0)-v_*\|_\xi\le M e^{-\omega t}\|v^0-v_*\|_\alpha\,,\qquad t  >  0\,.
\end{equation}
\end{thm}

\begin{rems}\label{RRa2'}   (i) In applications, it often holds that  $(\gamma_*, q_*) = (\gamma, q)$ and hence $\alpha_{\rm  crit}= \alpha_{\rm  crit}^*$, see e.g. Example~\ref{Exam0}  below  and~\cite[Example 4.1]{MRW25}.

 (ii) In the noncritical case $\alpha_{\rm crit} < \alpha = \gamma$ (assuming (a)) or  $\max\{\alpha_{\rm  crit},\,  \alpha_{\rm  crit}^*\} < \alpha \in [\gamma, \xi)$ (assuming (b)), 
the condition~\eqref{Vorcc} in Theorem~\ref{Tsemi}  can be omitted, 
and the claim remains valid (with a similar proof) provided  that~\eqref{stable'} is   replaced by
\begin{equation*}
\|v(t;v^0)-v_*\|_\alpha +t^{\mu}\|v(t;v^0)-v_*\|_\xi\le M e^{-\omega t}\|v^0-v_*\|_\alpha\,,\qquad t> 0\,,
\end{equation*} 
for some fixed (but arbitrary)  $\mu>\xi-\alpha$. 
For similar reasons as in Remark~\ref{RRa1}, we present  herein the (slightly) less general  but more concise result in Theorem~\ref{Tsemi}.
\end{rems}

\subsection*{Instability}  For the sake of completeness we also state  conditions sufficient for instability of an equilibrium solution~$v_*\in O_1\subset E_1$ 
to the quasilinear problem \eqref{EE} or the semilinear problem \eqref{SEP}.
Specifically, we assume that 
  \begin{subequations}\label{AsINST}
  \begin{equation}\label{a2a}
(A,f)\in C^{2-}\big(O_1,\kL(E_1,E_0)\times  E_0\big)
\end{equation}
with Fr\'echet derivatives
\begin{equation}\label{a2b}
 \p f(v_*),\, (\p A(v_*)[\cdot])v_* \in \kL(E_{\eta},E_0)\quad\text{for some $\eta\in[\beta,1)$}\,.
\end{equation}
Moreover, we require that the linearized operator 
$$
\bA= A(v_*)+(\p A(v_*)[\cdot])v_*+\p f(v_*)\in\mathcal{H}(E_1,E_0)
$$ 
satisfies  
\begin{equation}\label{a2c}
 \sigma_+(\bA):=\{\lambda\in \sigma(\bA)\,:\,  {\rm Re\,}\lambda>0\}\neq \emptyset\,,\qquad \inf \{{\rm Re\,}\lambda\,:\, \lambda\in \sigma_+(\bA)\}>0\,.
\end{equation}
  \end{subequations}
Conditions~\eqref{AsINST} guarantee the instability  of $v_*$   in the phase space $E_\alpha$: 

\begin{thm}[Instability] \label{T3}
 Assume \eqref{AsINST} and that the assumptions of Theorem~\ref{T0} are satisfied for the quasilinear problem~\eqref{EE} 
or, alternatively, the assumptions of Theorem~\ref{T0'} are satisfied for the semilinear problem~\eqref{SEP}. 
 
 Then, the equilibrium $v_*$ is unstable in $E_\alpha$. 
 More precisely, there exists a neighborhood $U$ of $v_*$ in $O_\alpha$ such that for each $n\in\N^*$ 
 there exists $v_{n}^0\in \bB_{E_\alpha}(v_*,1/n)\cap O_\alpha$ such that the corresponding solution $v(\cdot; v_n^0)$  to \eqref{EE} or \eqref{SEP} satisfies
 \[
 v(t; v_n^0)\not\in U\quad\text{for some $t\in(0,t^+(v_n^0))$.}
 \]
\end{thm}
The proof of this result is similar to the one in the  classical case $\xi= \beta$, see \cite[Theorem 1.4]{MW_MOFM20}, and is based on an instability result for fully-nonlinear parabolic problems established in \cite[Theorem 9.1.3]{L95}, where instability in $E_1$ is proven.  
\medskip

\subsection*{Outline} The outline of the paper is as follows. 
Section~\ref{Sec:2} is dedicated to the quasilinear problem~\eqref{EE}, where we provide  a proof of Theorem~\ref{T1} and establish the instability result stated in Theorem~\ref{T3}.
In Section~\ref{Sec:3} we  prove Theorem~\ref{Tsemi} for the semilinear problem~\eqref{SEP}.
In Section~\ref{Sec:5} we present examples that illustrate our previous findings both for the critical case  $\alpha= \alpha_{\rm  crit}$ and the noncritical case~$\alpha> \alpha_{\rm  crit}$, respectively,
including quasilinear problems with quadratic semilinearities, a parabolic-parabolic chemotaxis  system, and a quasilinear evolution equation in critical spaces with scaling invariance.

\section{Quasilinear Problem: Proof of Theorem~\ref{T1}}\label{Sec:2}

We shall use the framework of time-weighted spaces,  which is well suited to capturing the behavior of  solutions to~\eqref{EE} (and~\eqref{SEP}) near $0$.
Recall that, given a Banach space $E$,  $\mu\in\R$, and~${T>0}$, the space
\[
 C_\mu\big((0,T],E\big):=\big\{u\in C((0,T], E)\,:\, \text{$t^\mu  \|u(t)\|_E\to 0$ for $t\to0$}\big\} 
 \]
 is a Banach space with norm
\[
\|u\|_{C_\mu((0,T],E)}:=\sup_{t\in(0,T]}t^\mu\|u(t)\|_E\,. 
\]
Given $\omega>0$ and $\kappa\ge 1$,  we denote by~${\mathcal{H}(E_1,E_0;\kappa,\omega)}$ the class of all operators $\mathcal{A}\in\mathcal{L}(E_1,E_0)$ with the property
that $\omega-\mathcal{A}$ is an
isomorphism from $E_1$ onto~$E_0$ and
$$
\frac{1}{\kappa}\,\le\,\frac{\|(\mu-\mathcal{A})z\|_{0}}{\vert\mu\vert \,\| z\|_{0}+\|z\|_{1}}\,\le \, \kappa\ ,\qquad {\rm Re\,}
\mu\ge \omega\ ,\quad z\in E_1\setminus\{0\}\,.
$$
 Then $$\mathcal{H}(E_1,E_0)=\bigcup_{\omega>0\,,\,\kappa\ge 1} \mathcal{H}(E_1,E_0;\kappa,\omega)\,,$$
where $\mathcal{H}(E_1,E_0)$ is the set of generators of analytic semigroups on $E_0$ with domain $E_1$ (equipped with the graph norm).

For the proof of Theorem~\ref{T1} we need to recall the following local well-posedness result established in \cite{MW_PRSE,MRW25} 
(that, in turn,  serves as a basis for Theorem~\ref{T0}). In particular, it provides locally uniform estimates with respect to the initial value.

\begin{prop}\label{FF}
Assume~\eqref{AS1} and \eqref{AS2y}. 
Then, given $\ov v\in O_\alpha$, there exist constants~$r=r(\ov v)>0$ and $T=T(\ov v)>0$  such that for each $v^0\in\bar{\mathbb{B}}_{E_\alpha}(\ov v,r)\subset O_\alpha$ the Cauchy problem~\eqref{EE}  possesses a unique solution
\begin{equation}\label{F}
\begin{aligned}
v(\cdot;v^0)&\in  C\big((0,T],E_1\big)\cap C^1\big((0,T],E_0\big)\cap C\big([0,T],O_\alpha\big)\cap C_{\xi-\alpha}\big((0,T],E_\xi\big)
 \cap  C^{\alpha-\beta}\big([0,T],E_\beta\big)\,.
\end{aligned}
\end{equation}
In fact, there exists a constant $c=c(\ov v)>0$ such that, for $v^0,\, v^1\in \bar{\mathbb{B}}_{E_\alpha}(\ov v,r)$, it holds that
\begin{equation}\label{F1}
 \|v(\cdot;v^0)-v(\cdot;  v^1)\|_{C([0,T],E_\alpha)} + \|v(\cdot;v^0)-v(\cdot;  v^1)\|_{C_{\xi-\alpha}((0,T],E_\xi)}\leq  c\|v^0- v^1\|_\alpha\,.
\end{equation}
Moreover, there are constants $\omega>0$, $\kappa\ge 1$, $L>0$, and $\rho\in(0,\alpha-\beta)$ such that, for $v^0,\, v^1\in \bar{\mathbb{B}}_{E_\alpha}(\ov v,r)$,
\begin{equation}\label{Ff2}
 A(v(t;v^0))\in\mathcal{H}(E_1,E_0;\kappa,\omega)\,,\qquad t\in [0,T]\,,
\end{equation}
and
\begin{equation}\label{Ff2a}
 \|A(v(t;v^0))-A(v(s;v^0))\|_{\mathcal{L}(E_1,E_0)}\leq L(t-s)^\rho \,,\qquad 0\leq s\leq t\leq T \,,
\end{equation}
while
\begin{equation}\label{F2}
 \|A(v(\cdot;v^0))-A(v(\cdot;  v^1))\|_{C([0,T],\mathcal{L}(E_1,E_0))}\leq c\|v^0- v^1\|_\alpha\,.
\end{equation}
\end{prop}

\begin{proof}
This is established in \cite[Proposition~2.1]{MW_PRSE} for $\alpha_{\rm crit}<\alpha$ and in  \cite[Proposition~2.1]{MRW25} for $\alpha_{\rm  crit}=\alpha$, 
with the observation that in the noncritical case $\alpha_{\rm crit}<\alpha$ the regularity~\eqref{F} follows in view of \eqref{AS2y} by arguing as in the proof of Theorem~\ref{T0} above.
 In fact, \eqref{F1} is derived in the last step of the proof of \cite[Proposition~2.1]{MRW25} (see Eq.~(2.36)), respectively, in the noncritical case, \eqref{F1} follows from the estimates in the last part of the proof of \cite[Proposition~2.1]{MW_PRSE}, which can be sharpened using~\eqref{AS2y}.
 Moreover,~\eqref{Ff2}-\eqref{Ff2a} are stated in \cite[Eq.(2.8)]{MW_PRSE} respectively in \cite[Eq.(2.6)]{MRW25}. Finally, \eqref{F2} is a consequence of \eqref{a1b},~\eqref{F1}, 
 the embedding $E_\alpha\hookrightarrow E_\beta$, and the construction of the solutions (see \cite[Eq.(2.5)]{MW_PRSE} and \cite[Eq.(2.4)]{MRW25}).
\end{proof}

The arguments  that lead to~\eqref{F1} can be further exploited  in order to  derive related estimates in time-weighted spaces~$C_\mu((0,T], E_\zeta)$, with~$\zeta \in (\xi,1)$ and an appropriate exponent~$\mu>0$.

\begin{cor}\label{FFF}
Assume~\eqref{AS1} and \eqref{AS2y}. Given $\ov v\in O_\alpha$, let ~$r=r(\ov v)>0$ and $T=T(\ov v)>0  $ be the constants from Proposition~\ref{FF}
and chose $\zeta\in(\xi,1)$ and $\eta\in(0,\gamma)$.
Then, there is a constant $c_1=c_1(\ov v)>0$ such that, for $v^0,\, v^1\in \bar{\mathbb{B}}_{E_\alpha}(\ov v,r)$,
\begin{equation}\label{FFF1}
\|v(\cdot;v^0)-v(\cdot;  v^1)\|_{C_{\zeta-\alpha+\eta}((0,T],E_\zeta)}\leq c_1 \|v^0- v^1\|_\alpha \,.
\end{equation}
\end{cor}

\begin{proof}
To start, we infer from \eqref{Ff2}-\eqref{Ff2a} and \cite[II.Section~5]{LQPP} for 
 $v^0\in \bar{\mathbb{B}}_{E_\alpha}(\ov v,r)$ that $$A(v(\cdot;v^0))\in C^{\rho}([0,T],\mathcal{H}(E_1,E_0))$$ generates an evolution operator $U_{A(v(\cdot;v^0))}(t,s)$, $0\le s\le t\le T$, satisfying uniform stability estimates. More precisely,
\cite[II.Lemma 5.1.3]{LQPP}  ensures that there is a constant $N=N(\ov v)>0$  such that
\begin{equation}\label{g} 
(t-s)^{\zeta-\vartheta+\eta}\|U_{A(v(\cdot;v^0))}(t,s)\|_{\mathcal{L}(E_\vartheta,E_\zeta)}  \le N \,,\qquad 0\le s< t\le T  \,,
\end{equation}
for  $\vartheta\in \{\gamma,\alpha\}$ and $v^0\in \bar{\mathbb{B}}_{E_\alpha}(\ov v,r)$.
Moreover, we  deduce  from~\eqref{F2} and \cite[II.Lemma~5.1.4]{LQPP}  that
\begin{equation}\label{gg}
(t-s)^{\zeta-\vartheta}\|U_{A(v(\cdot;v^0))}(t,s)-U_{A(v(\cdot;v^1))}(t,s)\|_{\mathcal{L}(E_\vartheta,E_\zeta)}\le N \|v^0- v^1\|_\alpha\,,\quad 0\le s<t\le T\,,
\end{equation}
for  $\vartheta\in \{\gamma,\alpha\}$ and $v^0,\, v^1\in \bar{\mathbb{B}}_{E_\alpha}(\ov v,r)$.  
Invoking also~\eqref{AS2x} and~\eqref{F1} we find a constant $c_2=c_2(\ov v)>0$ such that
\begin{equation}\label{ggg0}
\|f(v(t;v^0))\|_{\gamma}\le c_2 t^{-q(\xi-\alpha)}\,,\quad 0<t\le T\,,
\end{equation}
and
\begin{equation}\label{ggg}
\|f(v(t;v^0))-f(v(t;v^1))\|_{\gamma}\le c_2 t^{-q(\xi-\alpha)}\|v^0- v^1\|_\alpha\,,\quad 0<t\le T\,.
\end{equation}
Consequently, since the unique solution to \eqref{EE} satisfies the variation-of-constant formula, i.e.
$$
v(t;v^0)=U_{A(v(\cdot;v^0))}(t,0)v^0+\int_0^t U_{A(v(\cdot;v^0))}(t,s)f(v(s;v^0))\,\rd s\,,\qquad t\in [0,T]\,,
$$
we deduce from \eqref{g}-\eqref{ggg} for $t\in (0,T]$   and $v^0,\, v^1\in \bar{\mathbb{B}}_{E_\alpha}(\ov v,r)$  that
\begin{align*}
\|v(t;v^0)-v(t;v^1)\|_\zeta&\leq
\|U_{A(v(\cdot;v^0))}(t,0)-U_{A(v(\cdot;v^1))}(t,0)\|_{\mathcal{L}(E_\alpha,E_\zeta)}\,\|v^0\|_\alpha\\
&\qquad +\|U_{A(v(\cdot;v^1))}(t,0)\|_{\mathcal{L}(E_\alpha,E_\zeta)}\,\|v^0-v^1\|_{\alpha} \\
&\qquad +\int_0^t \big\|U_{A(v(\cdot;v^0))}(t,s)-U_{A(v(\cdot;v^1))}(t,s)\big\|_{\mathcal{L}(E_\gamma,E_\zeta)}\, \|f(v(s;v^0))\|_\gamma\, \rd s\\
&\qquad+\int_0^t \big\|U_{A(v(\cdot;v^1))}(t,s)\big\|_{\mathcal{L}(E_\gamma,E_\zeta)}\, \|f(v(s;v^0))-f(v(s;v^1))\|_\gamma\,\rd s\\
&\le c t^{\alpha-\zeta-\eta}\|v^0- v^1\|_\alpha+c \int_0^t (t-s)^{\gamma-\zeta-\eta} s^{-q(\xi-\alpha)}\,\rd s\, \|v^0- v^1\|_\alpha\\
&= c \big[t^{\alpha-\zeta -\eta}+t^{1+\gamma-\zeta- \eta-q(\xi-\alpha)}\mathsf{B}(1+\gamma-\zeta- \eta,1-q(\xi-\alpha))\big] \|v^0- v^1\|_\alpha\,,
\end{align*}
with $\mathsf{B}$ denoting the Beta function. Recalling that by \eqref{X}-\eqref{XX} we have 
$$ 1+\gamma-\alpha-q(\xi-\alpha)\geq 0\,,$$ 
the assertion follows.
\end{proof}

\subsection*{Proof of Theorem~\ref{T1}}  Let the assumptions of Theorem~\ref{T1} be satisfied and fix $ 0<\omega<\bar\omega<\omega_0$. Since
\begin{equation*}
(A,f)\in C^{1-}\big(O_\xi,\mathcal{H}(E_1,E_0)\times E_\gamma\big)
\end{equation*}
and $\zeta\in (\xi,1)$, it follows from~\eqref{AS4} and \cite[Theorem~1.3]{MW_MOFM20} that there  exist $\ve_1>0$ and~$N\ge 1$ such that, 
 for each $w^0\in \ov\bB_{E_\zeta}(v_*,\ve_1)$, the solution to \eqref{EE} with initial value $v(0)=w^0$ exists globally in time and
\begin{equation}\label{stablexx}
\|v(t;w^0)-v_*\|_\zeta\le N e^{-\bar\omega t}\|w^0-v_*\|_\zeta\,,\qquad t\ge 0\,.
\end{equation}
Let $T(v_*)> 0$ and $r(v_*)>0$ be the constants from Proposition~\ref{FF} and fix an arbitrary $t_*\in (0,T(v_*))$.
 Then Corollary~\ref{FFF} implies that there exists $\ve_0\in (0,r(v_*))$  and a constant $c>0$ such that for each initial value $v^0\in \ov\bB_{E_\alpha}(v_*,\ve_0)$ we have $w^0:=v(t_*;v^0)\in \ov\bB_{E_\zeta}(v_*,\ve_1)$ and
\begin{equation}\label{stablex}
\|v(t_*;v^0)-v_*\|_\zeta\le c\|v^0-v_*\|_\alpha\,.
\end{equation}
Hence, for such  $v^0\in \ov\bB_{E_\alpha}(v_*,\ve_0)$,  it follows from $v(t;v^0)=v(t-t_*;v(t_*;v^0))$ for $t\in [t_*,t^+(v^0))$ (by uniqueness) and $t^+(v(t_*;v^0))=\infty$ that $t^+(v^0)=\infty$ and,  using~\eqref{stablexx}-\eqref{stablex},
\begin{align}
\|v(t;v^0)-v_*\|_\alpha&\le c \|v(t;v^0)-v_*\|_\xi  \le c \|v(t-t_*;v(t_*;v^0))-v_*\|_\zeta\nonumber\\
&\le cN e^{-\bar\omega (t-t_*)}\|v(t_*;v^0)-v_*\|_\alpha
\le c e^{-\bar\omega t}\|v^0-v_*\|_\alpha \label{df}
\end{align}
for $t\ge t_*$. 
Recalling that $\omega<\bar\omega$, we conclude from~\eqref{F1}  and~\eqref{df} that there exists    $M>0$  such that 
\begin{equation*}
\|v(t;v^0)-v_*\|_\alpha +t^{\xi-\alpha}\|v(t;v^0)-v_*\|_\xi\le M e^{-\omega t}\|v^0-v_*\|_\alpha\,,\qquad t> 0\,,
\end{equation*}
for each $v^0\in \ov\bB_{E_\alpha}(v_*,\ve_0)$, which proves Theorem~\ref{T1}.

\medskip

 Concerning the instability result Theorem~\ref{T3}, we note the following:

\subsection*{Proof of Theorem~\ref{T3}}
The proof is identically to that of \cite[Theorem 1.4]{MW_MOFM20}.
\qed

\section{Semilinear Problem: Proof of Theorem~\ref{Tsemi}}\label{Sec:3}

 This section is devoted to the proof of Theorem~\ref{Tsemi}. Under the assumption (a) of this theorem, the statement follows along the same lines as the proof of Theorem~\ref{T1}.
As a preliminary result, we therefore establish in Proposition~\ref{FFsemi} the counterpart of Proposition~\ref{FF} and Corollary~\ref{FFF} in the semilinear case.
 
 \begin{prop}\label{FFsemi}
Assume~\eqref{Vor}   with $\xi< 1$  and \eqref{Vorcc}.
Then, given $\ov v\in O_\alpha$, there exist~${r=r(\ov v)>0}$ and~$T=T(\ov v)>0$  such that, for each $v^0\in\bar{\mathbb{B}}_{E_\alpha}(\ov v,r)\subset O_\alpha$, the Cauchy problem~\eqref{SEP}  possesses a unique solution
\begin{equation}\label{Fsemi}
\begin{aligned}
v(\cdot;v^0)&\in  C\big((0,T],E_1\big)\cap C^1\big((0,T],E_0\big)\cap C\big([0,T],O_\alpha\big)\cap C_{\xi-\alpha}\big((0,T],E_\xi\big)\,.
\end{aligned}
\end{equation}
Moreover, there exists a constant $c=c(\ov v)>0$ such that, for $v^0,\, v^1\in \bar{\mathbb{B}}_{E_\alpha}(\ov v,r)$, it  holds that
\begin{equation}\label{F1semi}
 \|v(\cdot;v^0)-v(\cdot;  v^1)\|_{C([0,T],E_\alpha)}+\|v(\cdot;v^0)-v(\cdot;  v^1)\|_{C_{\xi-\alpha}((0,T],E_\xi)}  \leq c\|v^0- v^1\|_\alpha.
\end{equation}
Fixing an arbitrary $\zeta\in(\xi,1)$  it  additionally holds,  for $v^0,\, v^1\in \bar{\mathbb{B}}_{E_\alpha}(\ov v,r)$, that
\begin{equation}\label{F1semi2}
 \|v(\cdot;v^0)-v(\cdot;  v^1)\|_{C_{\zeta}((0,T],E_\zeta)}  \leq c\|v^0- v^1\|_\alpha.
\end{equation}
\end{prop}
\begin{proof}
The claims, up to the estimate~\eqref{F1semi2}, are established  in \cite[Proposition~3.1]{MW_PRSE} for $\alpha_{\rm crit} < \alpha$ and in \cite[Proposition~3.1]{MRW25} for $\alpha_{\rm crit} = \alpha$, 
with the observation that in the noncritical case $\alpha_{\rm crit} < \alpha$ the regularity property~\eqref{Fsemi}  
and the estimate~\eqref{F1semi} in~$C_{\xi-\alpha}((0,T],E_\xi)$  are obtained by also exploiting~\eqref{Vorcc} (see the proof of~\eqref{F1semi2} below).

In order to prove~\eqref{F1semi2}, we infer from \eqref{Vor3} and \eqref{F1semi}, that there exists a constant $c=c(\bar v)$ such that for  all $v^0,\, v^1\in \bar{\mathbb{B}}_{E_\alpha}(\ov v,r)$ we have
\begin{equation}\label{gggsemi}
\|f(v(t;v^0))-f(v(t;v^1))\|_{0}\leq \|f(v(t;v^0))-f(v(t;v^1))\|_{\gamma}\le c_2 t^{-q(\xi-\alpha)}\|v^0- v^1\|_\alpha\,,\quad 0<t\le T\,.
\end{equation}
Since solutions to \eqref{SEP} satisfy the variation of constant formula 
$$
v(t;v^0)=e^{tA}v^0+\int_0^t e^{(t-s)A}f(v(s;v^0))\,\rd s\,,\qquad t\in [0,T]\,,
$$
using \eqref{gggsemi} together with  the estimate $\|e^{(t-s)A}\|_{\kL(E_0,E_\zeta)}\leq C(t-s)^{-\zeta}$ for  $0\le s < t\le T$, 
we get for $t\in(0,T]$ and $v^0\in\bar{\mathbb{B}}_{E_\alpha}(\ov v,r)$ that
\begin{align*}
t^\zeta\|v(t;v^0)-v(t;v^1)\|_\zeta&\leq t^\zeta\|e^{tA}\|_{\mathcal{L}(E_0,E_\zeta)}\,\|v^0-v^1\|_{0} \\
&\qquad+t^\zeta\int_0^t \big\|e^{(t-s)A}\big\|_{\mathcal{L}(E_0,E_\zeta)}\, \|f(v(s;v^0))-f(v(s;v^1))\|_0\,\rd s\\
&\le c\|v^0- v^1\|_\alpha+c t^\zeta\int_0^t (t-s)^{-\zeta} s^{-q(\xi-\alpha)}\,\rd s\, \|v^0- v^1\|_\alpha\\
&= c \big[1+t^{1-q(\xi-\alpha)}\mathsf{B}(1-\zeta,1-q(\xi-\alpha))\big] \|v^0- v^1\|_\alpha\,.
\end{align*}
Above, $\mathsf{B}$ denotes again the Beta function and we used the  inequality $1-q(\xi-\alpha)>0$
which is a consequence of  \eqref{Xs}-\eqref{XXs}.
\end{proof}\medskip

 We conclude this section with the proof of Theorem~\ref{Tsemi}.

\subsection*{Proof of Theorem~\ref{Tsemi}} If the assumption  (a) holds, then the  proof follows from Proposition~\ref{FFsemi}, using a line of reasoning similar to that in the proof of Theorem~\ref{T1}.

 We now prove  Theorem~\ref{Tsemi}, assuming (b).
Since in this case $\alpha\geq \alpha_{\rm  crit}^*$, we have
\[
 \mu:=\xi-\alpha\leq \frac{1+\gamma_*-\alpha}{q_*}\,,
\]
the inequality turning into  a strict inequality  if $\alpha> \alpha_{\rm  crit}^*$, that is, if $\gamma_*\in (0, \gamma)$  or $\alpha=\gamma$ by assumption. 
If~${\gamma_*\in (0, \gamma)}$, we may thus choose $\gamma_0\in (0, \gamma_*)$ with $\mu q_*<1+\gamma_0-\alpha$, while we set $\gamma_0=\gamma_*$ for $\gamma_*\in\{0,\gamma\}$.
We then have in any case that
\[
\mu q_*\leq 1+\gamma_0-\alpha\,.
\]
Let  ${R\in(0,r_*)}$, see \eqref{MM4},   be chosen such that $\ov{\mathbb{B}}_{E_\alpha}(v_*,R)\subset O_\alpha$ and  fix
\[
\omega_0> \bar\omega>\omega>0\,.
\] 
 Since  $\mu q_*<1$ and  $\mu q_*\leq 1+\gamma_0-\alpha=1+\gamma_0+\mu-\xi$, the constant 
 \begin{align*}
c_0&:= 1+\big({\sf B}_{\mu q_*}+{\sf B}_{\mu (q_*-1)}\big)\Big(\sup_{r>0} r^{1+\gamma_0-\alpha-\mu q_*}e^{(\omega-\bar\omega)   r}\Big)
\end{align*}
with  $ {\sf B}_\theta :={\sf B}(\theta, 1-\mu q_*)$ for $\theta>0$ ($\mathsf{B}$ denotes again the Beta function), is well-defined.
Assumption \eqref{MM5} together with   \cite[II.Lemma~5.1.3]{LQPP}  
ensure  the existence of a constant~${M\geq 1}$ such that the strongly continuous analytic semigroup  $(e^{t\bA})_{t\geq0}$ generated by~$\bA$ satisfies
\begin{equation}\label{v1as}
\|e^{t\bA}\|_{\mathcal{L}(E_\theta)}+ t^{\theta-\vartheta_0}\|e^{t\bA}\|_{\mathcal{L}(E_\vartheta,E_\theta)}  \le  \frac{M}{4c_0}e^{- \bar\omega t}  \,, \qquad t>0\,,
\end{equation}
for $0\le \vartheta_0\le \vartheta\le \theta\le 1$ with $\vartheta_0<\vartheta$ if $0<\vartheta< \theta< 1$  or, thanks to~\eqref{Vorcc}, 
for~$(\vartheta,\theta)\in\{(\alpha,\xi), (\gamma,\alpha),(\gamma,\xi)\}$ with $\vartheta_0=\vartheta$.
 We may then  fix $L\in (0,R)$ such that 
\begin{equation}\label{L}
c_* M  L^{q_*-1}\le 1\,,
\end{equation}
where $c_*>0$ stems from \eqref{MM4}.
We now set
\[
\wh f(w):= f(w+v_*)-f(v_*)-\p f(v_*)w \,,\qquad w\in \wh O_\xi:=O_\xi-v_*\,,
\]
and note,  for $v^0\in \ov\bB_{E_\alpha}(v_*,L/M)$, that  $u:=v(\cdot;v^0)-v_*$
  is a strong  solution to the  evolution problem
\begin{equation}\label{4se}
  u'= \bA u+\wh f(u)\,,\quad t>0\,,\qquad u(0)=u^0 :=v^0-v_*\in\ov\bB_{E_\alpha}(0,L/M\big)\,.
\end{equation}
Set
\[
t_1:=\sup\big\{t\in (0,t^+(v^0))\,:\, \text{$\|u\|_{C_\mu ((0,t],E_\xi)}<L$\, and  \, $\|u\|_{C ([0,t],E_\alpha)}<R $}\big\} >0\,,
\]
 noticing $\|u\|_{C_\mu((0,t],E_\xi)}\to 0$ as $t\to 0$  and $\|u(0)\|_\alpha\leq L<R$.
Let $t\in (0,t_1)$. Then, $\|u(s)\|_\xi\le s^{-\mu} L$ for all~${s\in (0,t]}$  and $u$  satisfies in view of \eqref{solse} and \eqref{4se} the variation-of-constants formula
\[
u(\tau)=e^{\tau \bA}u^0+\int_0^\tau e^{(\tau-s)\bA}\wh f(u(s))\,\rd s\,,\qquad 0\le \tau\le t\,,
\]
with 
\begin{equation}\label{nonf}
\|\wh f (u(s))\|_{\gamma_*}\leq  c_*\|u(s)\|_\xi^{q_*}\leq c_*L^{q_*} s^{-\mu q_*},\qquad s\in(0,t]\,,
\end{equation}
see \eqref{MM4}. 
 In view of the estimate
\begin{align}\label{est4}
\sup_{t>0}\left(t^{a}\int_0^1(1-s)^{-b}\, e^{-\eta t(1-s)} \, s^{-\mu q_*}\,\rd s\right)
&\le \Big(\sup_{r>0}r^{a}e^{-\eta r}\Big)\,\mathsf{B}_{1-a-b}\,,\qquad 0\le a<1-b\,,\quad \eta>0,
\end{align} 
we deduce from \eqref{v1as} and \eqref{nonf} that
\begin{align}
 \|u(t)\|_\alpha&\leq   \|e^{t \bA}\|_{\kL(E_\alpha)}\|u^0\|_\alpha+\int_0^t \|e^{(t-s)\bA}\|_{\kL(E_{\gamma_*},E_\alpha)}\|\wh f(u(s))\|_{E_{\gamma_*}}\,\rd s\nonumber\\
&\le \frac{  M}{4} \|u^0\|_\alpha +  \frac{c_*  M   L^{q_*}}{4}\,,\label{bbbse1}
\end{align}
and similarly
\begin{align}
 t^{\mu}\|u(t)\|_\xi&\leq    t^{\mu}\|e^{t \bA}\|_{\kL(E_\alpha, E_\xi)}\|u^0\|_{\alpha}+t^{\mu}\int_0^t \|e^{(t-s)\bA}\|_{\kL(E_{\gamma_*},E_\xi)}\|\wh f(u(s))\|_{E_{\gamma_*}}\,\rd s\nonumber\\
&\le \frac{M}{4} \|u^0\|_\alpha +  \frac{c_* M   L^{q_*}}{4}\,.\label{bbbse2}
\end{align}
It now follows  from~\eqref{L}, \eqref{bbbse1}, and \eqref{bbbse2}  that
\begin{equation}\label{globalestimates}
\|u\|_{C_\mu((0,t],E_\xi)}\le \frac{L}{2} \qquad    \text{and} \qquad  \|u\|_{C ([0,t],E_\alpha)}\leq\frac{R}{2}
\end{equation}
for each $t\in (0,t_1)$, hence $t_1=t^+(v^0)$.
In particular,  the solution~${v(\cdot;v^0):[0,t^+(v^0))\to E_\alpha}$  to \eqref{SEP} satisfies, in view of~\eqref{globalestimates},
\begin{equation*} 
\liminf_{t\nearrow t^+(v^0)}\mathrm{dist}_{E_\alpha}\big( v(t;v^0),\partial O_\alpha\big)>0
\end{equation*}
and
\[
\|v(t;v^0)\|_\xi\leq \|v_*\|_\xi+Lt^{-\mu}\,,\qquad   t\in (0,t^+(v^0))\,.\]
 This property together with~\eqref{Vor3} ensures that
\begin{equation*} 
 \underset{t\nearrow t^+(v^0)}{\limsup}\|f(v(t;v^0))\|_{0}<\infty\,, 
\end{equation*}
and  Theorem~\ref{T0'} implies now that $t^+(v^0)=\infty$.

Define 
\[
z(t):=\sup_{\tau\in(0,t]} \big(  \|u(\tau)\|_\alpha+\tau^\mu  \|u(\tau)\|_\xi\big)e^{\omega \tau}\,,\qquad t>0\,.
\]
 Given $0<\tau<t$, we then have 
\[
\|u(\tau)\|_\xi^{q_*}\le L^{{q_*}-1} z(t)\, \tau^{-\mu q_*}\,e^{-\omega \tau}\,,
\]
and together with   \eqref{MM4}, \eqref{v1as}, and the latter estimate we deduce, analogously to~\eqref{bbbse1}-\eqref{bbbse2}, that 
\begin{align*}
   z(t)\le \frac{  M}{2} \,\|u^0\|_\alpha + \frac{c_* M   L^{q_*-1}}{2}\, z(t)
\end{align*}
and therefore, by the choice of $L$ from~\eqref{L},
\[
z(t)\le  M\|u^0\|_\alpha \,,\qquad t>0\,,
\]
that is,
\[
\|u(t)\|_\alpha+t^{\mu}\|u(t)\|_\xi\le M\,e^{-\omega t}\,\|u^0\|_\alpha \,,\qquad t>0\,.
\]
This completes the proof.
\qed

\section{Applications}\label{Sec:5}

In this section we apply our theory to various semilinear and quasilinear evolution problems of parabolic type, examining both critical and noncritical regimes, to explore the stability properties of their equilibrium solutions. 
The applications include quasilinear problems with quadratic semilinearities (see Example~\ref{Exam0}), a parabolic-parabolic chemotaxis system (Example~\ref{Exam1}), and a quasilinear evolution equation in critical spaces exhibiting scaling invariance (Example~\ref{Exam2}).

 \subsection{Quaslilinear Problems with Quadratic Semilinearity}\label{Exam0} 
To give a first flavor of our results we consider for densely embedded Banach spaces 
$E_1\hookrightarrow E_0$ and complex interpolation spaces 
$E_\theta:=[E_0,E_1]_\theta$  with $\theta\in (0,1)$, the quasilinear problem
\begin{equation}\label{QCP}
u'=A(u)u +Q(u,u)\,,\quad t>0\,,\qquad u(0)=u^0\,,
\end{equation}
with
\begin{subequations}\label{E1}
\begin{equation}\label{w1}
A\in C^{1-}\big(E_\beta,\mathcal{H}(E_1,E_0) \big)
\end{equation}
and a quadratic bilinear term
\begin{equation}\label{w2}
Q\in \mathcal{L}^2\big(E_\xi,E_\gamma \big)\,.
\end{equation}
 We are interested in the stability properties of the  equilibrium solution $u_*=0\in E_1$ to \eqref{QCP}.
 In the context of \eqref{QCP}, it is  appropriate to choose
\begin{equation}\label{w3}
0< \gamma< \beta< \xi<1\,,\qquad q:=2\,,\qquad 2\xi-1-\gamma\leq \alpha\in (\beta,\xi)\,.
\end{equation}
 In the semilinear case when $A(u)=A\in \mathcal{H}(E_1,E_0)$, we put $\gamma_*:=\gamma$ and $q_*:=q=2$,
noticing that thus~$\alpha_{\rm crit}=\alpha_{\rm crit}^*=2\xi-1-\gamma$.
Then, $f:E_\xi\to E_\gamma$ with  $f(u):=Q(u,u)$ satisfies~\eqref{chris}  and~\eqref{MM4}.
\end{subequations}
 Consequently, we obtain 
from  Remark~\ref{RR1}, Theorem~\ref{T1}, Theorem~\ref{Tsemi}, and Theorem~\ref{T3}:

\begin{thm}\label{T4x}
Assume \eqref{E1}. If $s(A(0))<0$, then $u_*=0$ is asymptotically 
exponentially stable in $E_\alpha$ for the Cauchy problem~\eqref{QCP}. 
If $A$ is Fr\'echet differentiable with Lipschitz continuous derivative and
\begin{equation*}
  \sigma_+(A(0)):=\{\lambda\in \sigma(A(0))\,:\,  {\rm 
Re\,}\lambda>0\}\neq \emptyset\,,\qquad \inf \{{\rm Re\,}\lambda\,:\, 
\lambda\in \sigma_+(A(0))\}>0\,,
\end{equation*}
then $u_*=0$ is unstable in $E_\alpha$  for the Cauchy 
problem~\eqref{QCP}.  In the semilinear case when $A$ is independent of 
$u$, the same results hold if~\eqref{w3} is replaced by
\begin{equation*} 
0\le \gamma<  \xi\le 1\,,\quad (\gamma,\xi)\not=(0,1)\,,
\end{equation*}
 and either
\[
 2\xi-1-\gamma< \alpha\in 
[\gamma,\xi)\qquad\text{or}\qquad 2\xi-1-\gamma= \alpha\in 
(\gamma,\xi)\,.
\]
\end{thm}

Of course, as pointed out in Remark~\ref{RR1}, also other interpolation functors  than the complex functor~$[\cdot,\cdot]_\theta$ (such as real or continuous interpolation functors)
can be considered.\\
 
In the critical case $2\xi-1-\gamma= \alpha\in (\gamma,\xi)$,
Theorem~\ref{T4x} has been established previously in the particular context of  a semilinear asymptotic model for  atmospheric flows describing morning glory clouds, see \cite[Theorem~4.1]{MRW25}. 
 It is worth noting that Theorem~\ref{T4x} yields an exponential stability result which, to the best of our knowledge, cannot be obtained via the classical approach~\cite{MW_MOFM20, Amann_Teubner, D89}.   
For similar results for semilinear parabolic problems with general superlinear nonlinearities we refer to \cite[Corollary~1.4]{MRW25} and \cite[Example~4.2]{MRW25}. 
We also refer to \cite[Corollary~2.2]{PSW18} for a related exponential stability result in the semilinear case  within the framework of maximal $L_p$-regularity.

 \subsection{A Parabolic-Parabolic Chemotaxis System}\label{Exam1} We present an application of the exponential stability result in the semilinear case, cf. Theorem~\ref{Tsemi}, and of the instability result in Theorem \ref{T3} in the context of a
  parabolic-parabolic chemotaxis system
with logistic source, see e.g. \cite{Wi10, BW16}, 
\begin{subequations}\label{PP0}
\begin{align}
\partial_t u&=\mathrm{div}\big(\nabla u- \chi u\nabla v) +\kappa u(1-u)\,,&& t>0\,,\quad x\in \Omega\label{oi0}\,,\\
\partial_t v&=\Delta v + u- v\,,&& t>0\,,\quad x\in \Omega\,,\label{oi1}
\end{align}
where $\kappa$ and $\chi$ are positive constants,  subject to the initial conditions
\begin{align}\label{initcon}
u(0,x)=u^0(x)\,, \quad v(0,x)=v^0(x)\,,\qquad  x\in \Omega\,,
\end{align}
and homogeneous Neumann boundary conditions
\begin{align}\label{boundcon}
 \partial_\nu u= \partial_\nu v&= 0 \qquad  \text{on } \partial\Omega\,.
\end{align}
\end{subequations}
The functions $u^0,\,  v^0:\Omega\to\R$ are given   and $\Omega\subset\R^n$,  $n\geq 1$, is a smooth bounded domain with outward unit normal $\nu$.
Some constants from \cite{Wi10, BW16}, which are qualitatively irrelevant to the analysis, have been replaced by $1$.

We observe that problem \eqref{PP0} has exactly two constant  equilibrium solutions, namely
\begin{equation*}
(u_1,v_1)=(0,0)\qquad\text{and}\qquad (u_2,v_2)=(1,1)\,.
\end{equation*} 
It is well-known that for sufficiently smooth non-negative initial 
data, \eqref{PP0} possesses  a unique bounded global classical solution provided that  $\Omega$ is convex and $\kappa$ is sufficiently large, see e.g. \cite{Wi10}. 
We prove herein that~\eqref{PP0} is locally well-posed for more general initial data and that the equilibrium solution~${(u_1,v_1)=(0,0)}$ is always unstable.
Moreover, under certain restrictions  on the coefficients~$\chi$ and~$\kappa$,
we establish the exponential  stability of the equilibrium $(u_2,v_2)=(1,1)$. 
The advantages of the parabolic theory based on time-weighted function spaces in the context of \eqref{PP0} are highlighted in~\cite{MW_PRSE}, 
 where local well-posedness is established for a more general class of models, assuming initial data $ u_0 \in L_2(\Omega)$ in the physically relevant dimensions~${ n \leq 3}$. 
 Moreover, it is shown in~\cite{MW_PRSE} that global $L_2$-bounds on $u$ suffice to guarantee the global existence of solutions to~\eqref{PP0}. 
 In this context, our result in Theorem~\ref{TA1} complements the analysis in~\cite{MW_PRSE} by providing a corresponding stability result within the same low-regularity functional framework.

Before presenting  Theorem~\ref{TA1}, we  define for $p\in (1,\infty)$  the Banach spaces
$$
F_0:=L_p(\Omega)\,,\qquad F_1:= W_{p,N}^{2}(\Omega)=H_{p,N}^{2}(\Omega)=\{v\in H_{p}^{2}(\Omega)\,:\,  \p_\nu v=0 \text{ on } \partial\Omega\}\,, 
$$
and set, see \cite[\S 4]{Amann_Teubner},
$$
B_0:=\Delta_N:= \Delta\big|_{W_{p,N}^{2}(\Omega)}\in \mathcal{H}\big(W_{p,N}^{2}(\Omega),L_p(\Omega)\big)\,.
$$
Let
$$
\big\{(F_\theta,B_\theta)\,:\, -1\le \theta<\infty\big\}\ 
$$
be the interpolation-extrapolation scale generated by $(F_0,B_0)$ and the 
complex interpolation functor~$[\cdot,\cdot]_\theta$ (see \cite[\S 6]{Amann_Teubner} and \cite[\S V.1]{LQPP}). Then,
\begin{equation}\label{f1}
B_\theta\in \mathcal{H}(F_{1+\theta},F_\theta)\,,\qquad -1\le \theta<\infty\,,
\end{equation}
where
\begin{equation}\label{f2}
 F_\theta\doteq H^{2\theta}_{p,N}(\Omega):=\left\{\begin{array}{ll} \{v\in H_{p}^{2\theta}(\Omega) \,:\, \p_\nu v=0 \text{ on } 
 \partial\Omega\}\,, &1+\frac{1}{p}<2\theta<3+\frac{1}{p} \,,\\[3pt]
	 H_{p}^{2\theta}(\Omega)\,, & -1+\frac{1}{p}< 2\theta<1 +\frac{1}{p}\,,\end{array} \right.
\end{equation}
see \cite[Theorem~7.1; Equation (7.5)]{Amann_Teubner}\footnote{This property  is stated in \cite{Amann_Teubner} for
   $-1+\frac{1}{p} < 2\theta\le 2$. However,
since $(1-\Delta_N)^{-1}\in \mathcal{L}(H_{p}^{2\theta-2}(\Omega),H_{p}^{2\theta}(\Omega))$ for $2<2\theta<3+1/p$, see \cite[Theorem 5.5.1]{Tr78}, we obtain the full range in \eqref{f2}.}.
Moreover, since  $\Delta_N-1$ has bounded imaginary powers, see e.g. from~\cite[III.~Examples 4.7.3~(d)]{LQPP}, we  infer from  \cite[Remarks~6.1~(d)]{Amann_Teubner} that
\begin{equation}\label{f3}
[F_\beta,F_\alpha]_\theta\doteq F_{(1-\theta)\beta+\theta\alpha}\,,\qquad -1\leq \beta<\alpha\,,\quad \theta\in(0,1)\,.
\end{equation}

 Theorem~\ref{TA1} below provides the aforementioned stability result for  problem \eqref{PP0}, which,   despite the quasilinear term $\mathrm{div}\big(u\nabla v)$ in \eqref{oi0}, is treated
  as a semilinear evolution problem in the subcritical regime when $\alpha_{\rm  crit}<\alpha$, see \eqref{XXs}.
\begin{thm}\label{TA1}
Fix  $\kappa,\,\chi\in(0,\infty)$, $n\geq 1$,  $p\in (1,\infty) $ with $p>n/2$.
 Then,   \eqref{PP0} generates a semiflow on~$L_p(\Omega)\times H_{2p}^{1}(\Omega)$. 
 Moreover, the following hold:
\begin{itemize}
\item[(i)] If the maximal existence time $t^+\in(0,\infty]$ of the solution  $(u,v)$ to \eqref{PP0} is finite, then
\[
\limsup_{t\nearrow t^+}\|u(t)\|_{L_p}=\infty\,.
\] 
\item[(ii)] The equilibrium solution $(u_1,v_1)=(0,0)$ is unstable in $L_p(\Omega)\times H_{2p}^{1}(\Omega)$.
\item[(iii)] If $\chi\leq 2$ and $\kappa>1/4$, then $(u_2,v_2)=(1,1)$ is  exponentially stable in $L_p(\Omega)\times H_{2p}^{1}(\Omega)$. More precisely, given   
\[
 0< 2\ve< \min\Big\{1-\frac{1}{p},\,  1-\frac{n}{2p}\Big\}\qquad \text{and}\qquad \mu>\frac{1-\ve}{2}\,,
\] 
 there exist constants $r,\,\omega\in (0,1)$ and $M\geq 1$
 such that for all initial data~${(u_0,v_0)\in L_p(\Omega)\times H_{2p}^{1}(\Omega)}$ with~${\|(u_0,v_0)\|_{L_p\times H^{1}_{2p}}\leq r}$, the solution~$(u,v)$ to~\eqref{PP0}
    is globally defined and
 \begin{equation}\label{estex1}
  \|(u(t)-1,v(t)-1) \|_{L_p\times H^{1}_{2p}}+t^{\mu}\|(u(t)-1,v(t)-1)\|_{H^{1-\ve}_p\times H^{2-\ve}_{2p}}\le M e^{-\omega t}\,\|(u_0,v_0)\|_{L_p\times H^{1}_{2p}} \,,\qquad t>0\,.
 \end{equation}
\end{itemize} 
\end{thm}
\begin{proof} The local well-posedness of \eqref{PP0} (together with (i)) has been established in \cite[Theorem~5.1]{MW_PRSE} 
in a slightly more general context and we therefore only sketch the proof of this result.
Set 
\begin{align*}
&E_0:= H_{p,N}^{-2\ve}(\Omega)\times H_{2p,N}^{1-2\ve}(\Omega)\,, &E_1:= H_{p,N}^{2-2\ve}(\Omega)\times H_{2p,N}^{3-2\ve}(\Omega) \,,
\end{align*}
so that, by \eqref{f2}-\eqref{f3},
$$
E_\theta=H_{p,N}^{2\theta-2\ve}(\Omega)\times H_{q,N}^{1+2\theta-2\ve}(\Omega)\,,\qquad 2\theta\in [0,2]\setminus\{2\ve+ 1+1/p\,,\, 2\ve+1/2p\}\,.
$$
Choosing
\[
q:=2\qquad\text{and}\qquad 0<  \gamma:=\frac{\ve}{3}<\alpha:=\ve<\xi:=\frac{1+\ve}{2}<1\,,
\]
we have
\begin{equation*} 
 \alpha_{\rm crit}=2\xi-1-\gamma=\frac{2\ve}{3}<\alpha\in (\gamma,\xi) \,, 
\end{equation*}
see \eqref{Xs}, and
\[
E_\xi= H_{p,N}^{1- \ve}(\Omega)\times H_{2p,N}^{2- \ve}(\Omega)\hookrightarrow E_\alpha= L_p(\Omega)\times H_{2p}^{1}(\Omega)\hookrightarrow E_\gamma= H_{p,N}^{- 4\ve/3}(\Omega)\times H_{2p,N}^{1- 4\ve/3}(\Omega)\,.
\]
Since  $H_{p,N}^{2-2\ve}(\Omega)\hookrightarrow H_{2p,N}^{1-2\ve}(\Omega),$ we obtain from  \cite[I.~Theorem 1.6.1]{LQPP}   and  \eqref{f1}-\eqref{f2}   that 
$$
A:=\begin{pmatrix}
\Delta+\kappa&0\\[1ex]
1&\Delta-1
\end{pmatrix}\in \mathcal{H}(E_1,E_0)\,.
$$
Let $f:E_\xi\to E_\gamma$ be given by 
\[
f(w): =-\big(\chi\mathrm{div}\big( u\nabla v) +\kappa u^2,\,0\big)\,,\qquad w=(u,v)\in E_\xi.
\] 
Using the  continuity of the multiplications 
\[
H_{p,N}^{1- \ve}(\Omega)\bullet H_{2p,N}^{1- \ve}(\Omega)\longrightarrow H_{p,N}^{1- 4\ve/3}(\Omega) \qquad\text{and}\qquad
H_{p,N}^{1- \ve}(\Omega)\bullet H_{p,N}^{1- \ve}(\Omega)\longrightarrow L_p(\Omega)\,,
\]
 see \cite[Theorem~4.1]{AmannMult},  it is not difficult to conclude that   there is a constant $C>0$ such that
\begin{equation}\label{condf}
\|f(w)-f(\bar w)\|_{E_\gamma}\le C\big[\|w \|_{E_\xi}+\|\bar w\|_{E_\xi}\big] \|w-\bar w\|_{E_\xi} \,,\qquad w,\, \bar w\in E_\xi\,.
\end{equation}

 The local well-posedness of \eqref{PP0} and the blow-up criterion  (i)  follow by using \cite[Theorem 1.2]{MW_PRSE} as in the proof of \cite[Theorem~5.1]{MW_PRSE}.

In order to address the stability properties of the equilibria $(0,0)$ and $(1,1)$ we note that~${f\in C^{2-}(E_\xi, E_\gamma)}$ with 
 \[
\p f(\bar w) w= -\big(\chi\mathrm{div}\big( u\nabla \bar v+\bar u\nabla v) +2\kappa u \bar u\,,\,0 \big)\,,\qquad w=(u,v)\,, \,\bar w=(\bar u,\bar v)\in E_\xi\,.
 \]
 Moreover,  since~$A+\p f(w)\in \mathcal{H}(E_1,E_0)$ for all $w\in E_\xi$ with compact embedding~${E_1\hookrightarrow E_0}$, 
the spectrum   of~${A+\p f(w)}$ consists entirely of isolated  eigenvalues with finite algebraic multiplicities~\cite[Theorem~III.6.29]{Ka95}. 

For $w_1=(0,0)$ we have $\p f(w_1)=0$ and the linearization $A+\p f(w_1)=A$ has the positive eigenvalue~${\lambda=\kappa}$  (with a constant eigenvector). 
 Recalling Remark~\ref{RR1} and \eqref{f3}, we are in a position to apply Theorem \ref{T3} and deduce that $(0,0)$ is an unstable equilibrium for \eqref{PP0}.

For $w_2=(1,1)$  we have 
 \[
\p f(w_2) w=\big(\!-\chi\Delta v -2\kappa u \,,\,0 \big)\,,\qquad w=(u,v)\in E_\xi\,,
 \]
and, in order to apply Theorem~\ref{Tsemi}, it remains to verify that~\eqref{MM5} is  satisfied in the context of \eqref{PP0}. 
 Let thus $\lambda={\rm Re\,}\lambda+i\,{\rm Im\,}\lambda\in\mathbb{C}$ be an eigenvalue of~${A+\p f(w_2)}$ with eigenvector $0\neq w=(u,v)\in H^2_{ N}( \Omega)^2$ 
 by elliptic regularity, see e.g. \cite[Theorem~15.2]{ADN54}. 
It then holds 
\begin{align}
\Delta u-\chi \Delta v-\kappa u&=\lambda u\quad \text{in $\Omega$}\,,\label{EWE1}\\
\Delta   v+u-  v&=\lambda v\quad \text{in $\Omega$}\,. \label{EWE2}
\end{align}
Let $u=u_1+iu_2$ and  $v=v_1+iv_2$.
 Testing the real part of  \eqref{EWE1} with $u_1$ and the imaginary part with $u_2$ and proceeding in the same way with \eqref{EWE2} (where we test with $v_1$ and $v_2$, respectively), we arrive at
\begin{align*}
-\sum_{i=1}^2\int_{\Omega}\big[\big(|\nabla u_i|^2-\chi \nabla u_i\cdot\nabla v_i+|\nabla v_i|^2\big)+\big(\kappa| u_i|^2-u_i v_i+ v_i^2\big)\big]\,\rd x={\rm Re\,}\lambda\int_{\Omega}(|u|^2+|v|^2)\,\rd x\,.
\end{align*}
Since $\chi\leq 2$ an $\kappa>1/4$, Young's inequality  and the observation that $v\neq 0$ (otherwise $w=0$  by \eqref{EWE2}) leads us now to 
\begin{align*}
{\rm Re\,}\lambda\int_{\Omega}(|u|^2+|v|^2)\,\rd x\leq \Big(\frac{1}{4\kappa}-1\Big) \int_{\Omega}|v|^2\,\rd x<0\,,
\end{align*}
hence ${\rm Re\,}\lambda<0$, which proves~\eqref{MM5}.
Assertion (iii) is now a direct consequence of Theorem~\ref{Tsemi}.
\end{proof}

\subsection{A Quasilinear Problem with Scaling Invariance}\label{Exam2} We shall apply Theorem~\ref{T1} in the context of a quasilinear evolution equation from  \cite{PSW18, QS19}:
\begin{subequations}\label{Ex1}
\begin{equation}\label{Ex1a}
 \partial_tu={\rm div}(a(u)\nabla u) +|\nabla u|^{\kappa}\qquad \text{in $\Omega  $\,,\, $t>0 $\,,} 
\end{equation}
subject to homogeneous Dirichlet boundary conditions
\begin{equation}\label{Ex1b}
  u=0\qquad \text{on $\p\Omega$\,, \, $t>0$\,,} \\
\end{equation}
and the initial condition
\begin{equation}\label{Ex1c}
u(0)=u_0\,,
\end{equation}
\end{subequations}
where ${\kappa>3}$, $u_0:\Omega\to\R$ is a given function,  and $\Omega\subset\R^n$  with $n\geq 1$ is a smooth bounded domain.

In order to recast \eqref{Ex1}   in an appropriate framework we set for $p\in(1,\infty)$ 
\begin{equation*}
 H_{p,D}^{2\theta}(\Omega):=\left\{\begin{array}{ll} \{v\in H_{p}^{2\theta}(\Omega) \,:\,  v=0 \text{ on } 
 \partial\Omega\}\,, &\frac{1}{p}<2\theta\leq 2 \,,\\[3pt]
	 H_{p}^{2\theta}(\Omega)\,, & -2+\frac{1}{p}< 2\theta< \frac{1}{p}\,.\end{array} \right.
\end{equation*}
As observed  in \cite{PSW18, MRW25},  the space~$H^{s_c}_{p,D}(\Omega)$ with
\[
s_c:=\frac{n}{p}+\frac{\kappa-2}{\kappa-1}
\] 
can be identified (via a scaling invariance argument) as a critical space for  \eqref{Ex1}.
The next result,  which is new to the best of our knowledge, establishes the exponential stability of the zero solution to \eqref{Ex1} in this critical space.
For  a related exponential stability result in a critical Besov space we refer to \cite[Example 2]{PSW18}.
\begin{thm}\label{T:A3}
Let $\kappa>3$ and let $a\in C^{1}(\R)$ be a  strictly positive function
with uniformly Lipschitz continuous derivative.
We choose $p\in (2n,(\kappa-1)n)$  with $p \neq (n-1)(\kappa-1)$ and~$\tau\in(0,1)$ such that
\begin{equation*}
\frac{1}{2}<2\tau<1-\frac{n}{p} 
\end{equation*}
and set
\begin{equation*}
 0<\bar s:=2\tau+ \frac{n}{p}<s_c<s:=1+\frac{n(\kappa-1)}{p\kappa}<  2-2\tau\,,
\end{equation*}
as well as
\begin{equation*}
 \mu:=\frac{1}{2(\kappa-1)}-\frac{n}{2p\kappa}\in(0,1)\,.
\end{equation*}
Then, \eqref{Ex1} is locally well-posed in  $H^{s_c}_{p,D}(\Omega)$.
  Moreover, there exist constants $r,\,\omega\in (0,1)$ and $M\geq 1$ such that for all $\|u_0\|_{H^{s_c}_{p,D}}\leq r$ the solution~${u=u(\cdot; u_0)}$ to~\eqref{Ex1}
    is globally defined and
 \begin{equation}\label{estex3}
  \|u(t)\|_{H^{s_c}_{p}}+t^{\mu}\|u(t)\|_{H^{s}_p}\le M\, e^{-\omega t}\,\|u_0\|_{H^{s_c}_p} \,,\qquad t>0\,.
 \end{equation}
\end{thm}
\begin{proof} The local well-posedness result follows by arguing as in \cite[Example 3]{MRW25}, where homogeneous Neumann conditions were considered instead.
Therefore we only sketch the proof of the local well-posedness result. Set
$$
F_0:=L_p(\Omega)\,,\qquad F_1:= W_{p,D}^{2}(\Omega)=H_{p,D}^{2}(\Omega) \,,
$$
and note  from \cite[\S 4]{Amann_Teubner} that 
$$
B_0:=\Delta_D:=\Delta|_{W_{p,D}^{2}(\Omega)}\in \mathcal{H}\big(W_{p,D}^{2}(\Omega),L_p(\Omega)\big)\,.
$$
Let further
$$
\big\{(F_\theta,B_\theta)\,:\, -1\le \theta<\infty\big\}\ 
$$
be the interpolation-extrapolation scale generated by $(F_0,B_0)$ and the 
complex interpolation functor~$[\cdot,\cdot]_\theta$ (see \cite[\S 6]{Amann_Teubner} and \cite[\S V.1]{LQPP}), that is,
\begin{equation}\label{f2x}
B_\theta\in \mathcal{H}(F_{1+\theta},F_\theta)\,,\quad -1\le \theta<\infty\,,
\end{equation}
with (see \cite[Theorem~7.1; Equation (7.5)]{Amann_Teubner})
\begin{equation}\label{f2N}
 F_\theta\doteq H_{p,D}^{2\theta}(\Omega)\,,\qquad   2\theta\in\Big(-2+\frac{1}{p},2\Big]\setminus\Big\{\frac{1}{p}\Big\}\,.
\end{equation}
Moreover, since
$\Delta_D$ has bounded imaginary powers (see~\cite[III.~Examples 4.7.3~(d)]{LQPP}), we infer from \cite[Remarks~6.1~(d)]{Amann_Teubner} that
\begin{equation}\label{f3N}
 [F_\beta,F_\alpha]_\theta\doteq F_{(1-\theta)\beta+\theta\alpha}\,,\qquad -1\leq \beta<\alpha\,,\quad \theta\in(0,1)\,.
\end{equation}
With  $p$ and $\tau$ fixed   in the statement, we set
$$
E_\theta:= H^{2\theta-2\tau}_{p,D}(\Omega)\,,\qquad 2\tau+\frac{1}{p}\neq 2\theta\in[0,2]\,,
$$
 and point out that  $E_\theta=[E_0,E_1]_\theta$ and that
 none of the constants $\bar s$, $s$, and $s_c$  is equal to $1/p$. 
Let
 $$
 q:=\kappa>3
$$
and 
\begin{equation} \label{bnnn1} 
 0<\gamma:=\tau<\beta:= \tau+ \frac{\bar s }{2}< \alpha :=\tau +\frac{s}{2}<1\,,
\end{equation}
where
$$
\alpha_{\rm  crit}=\frac{\kappa \xi-1-\gamma}{\kappa-1}=\tau+ \frac{s_c}{2} \in(\beta,\xi)\,.
$$
Note that
\[
E_\xi=H^{s}_{p,D}(\Omega)\hookrightarrow E_\alpha=H^{s_c}_{p,D}(\Omega)\hookrightarrow E_\beta=H^{\bar s}_{p,D}(\Omega)\hookrightarrow E_\gamma=L_{p}(\Omega)\,.
\]
We may now formulate \eqref{Ex1} as the  quasilinear  evolution problem
\begin{equation}\label{QEE2}
u'=A(u)u+f(u)\,,\quad t>0\,,\qquad u(0)=u_0\,,
\end{equation} 
where $A:E_\beta\to\kL(E_1,E_0)$ is  defined by 
\[
A(u)v:={\rm div} (a(u)\nabla v)\,,\qquad  v\in E_1\,,\quad u\in E_\beta\,,
\]
and $f:E_\xi\to E_\gamma$  is defined by
\[
f(u):=|\nabla u|^{\kappa}\,,\qquad  u\in E_\xi\,.
\] 
As shown in \cite[Example 3]{MRW25},
\begin{align}\label{Lips}
\|f(u)-f(v)\|_{E_\gamma}\leq \kappa (\|u\|_{E_\xi}^{\kappa-1}+\|v\|_{E_\xi}^{\kappa-1})\|u-v\|_{E_\xi}\,,\qquad u,\, v\in E_\xi\,,
\end{align}
 and moreover $A\in C^{1-}(E_\beta,\mathcal{H}(E_1,E_0))$.
The latter property together with \eqref{f2N}-\eqref{bnnn1} and  \eqref{Lips} enables us to apply Theorem~\ref{T0} 
to~\eqref{QEE2}  and deduce the local well-posedness of this   problem in~$ H^{s_c}_{p,D}(\Omega)$.

We next verify \eqref{AS4} for  the stationary solution $v_*:=0\in E_1$. 
To this end we first infer from \eqref{Lips} that the map~$f:E_\xi\to E_\gamma$ is Fr\'echet differentiable in $0$ with $\p f(0)=0$.
Since the spectral bound 
\[
s(\Delta_D)=\sup\{\mathrm{Re}\, \lambda : \lambda\in\sigma(\Delta_D)\}
\] 
is negative,  assumptions~\eqref{AS4} are satisfied.
We are thus in a position to apply Theorem~\ref{T1} and establish in this way the exponential stability of the zero solution to~\eqref{QEE2} in~$ H^{s_c}_{p,D}(\Omega)$, see \eqref{estex3}.
\end{proof}



\section*{Acknowledgement}
We express our gratitude to the referees for their careful reading of the manuscript and for their valuable suggestions, which have helped to improve the paper.

\bibliographystyle{siam}
\bibliography{Literature}
\end{document}